\documentclass[a4paper,reqno,12pt]{amsart} 
\usepackage{amsmath} 

\usepackage{amssymb}      

\usepackage{yhmath}
\usepackage{mathdots}
\usepackage{MnSymbol}

\usepackage{amsthm}      

\usepackage{tikz,amsmath}
\usetikzlibrary{arrows}

\usepackage{epsfig}      

\usepackage{graphicx}
\usepackage[normalem]{ulem}
\usepackage[all,line,
]{xy} 
\CompileMatrices 

\newcommand{\Ad}{\operatorname{Ad}}

\newcommand{\id}{\operatorname{id}} 
\newcommand{\Aut}{\operatorname{Aut}}

\newcommand{\diag}{\operatorname{diag}}

\newcommand{\Span}{\operatorname{Span}}

\newcommand{\Sp}{\operatorname{Sp}}

   \theoremstyle{plain}
   \newtheorem{thm}{Theorem}[section]
   \newtheorem{prop}[thm]{Proposition}
   \newtheorem{lemma}[thm]{Lemma}  
   \newtheorem{cor}[thm]{Corollary}
   \theoremstyle{definition}

   \newtheorem{example}[thm]{Example}
   \theoremstyle{remark}
   
   \newtheorem{remark}[thm]{Remark}

\usepackage{mdframed,xcolor}

\definecolor{mybgcolor}{gray}{0.8}
\definecolor{myframecolor}{rgb}{.647,.129,.149}

\mdfdefinestyle{mystyle}{
  usetwoside=false,
  skipabove=0.6em plus 0.8em minus 0.2em,
  skipbelow=0.6em plus 0.8em minus 0.2em,
  innerleftmargin=.25em,
  innerrightmargin=0.25em,
  innertopmargin=0.25em,
  innerbottommargin=0.25em,
  leftmargin=-.75em,
  rightmargin=-0em,
  topline=false,
  rightline=false,
  bottomline=false,
  leftline=false,
  backgroundcolor=mybgcolor,
  splittopskip=0.75em,
  splitbottomskip=0.25em,
  innerleftmargin=0.5em,
  leftline=true,
  linecolor=myframecolor,
  linewidth=0.25em,
}

\newmdenv[style=mystyle]{important}

\usepackage{lipsum}

   \numberwithin{equation}{section}

        \date{\today}

\title[KMS states]{KMS states on crossed products by abelian groups}

\author{Johannes Christensen and Klaus Thomsen}

\date{\today}

\email{ johannes@math.au.dk, matkt@math.au.dk}
\address{Department of Mathematics, Aarhus University, Ny Munkegade, 8000 Aarhus C, Denmark}

\begin{document}

\maketitle

\section{Introduction}

In recent years there has been a quite steep increase in the number of papers on KMS states for one-parameter flows on $C^*$-algebras. Most have been concerned with specific examples, and as a result we have now a large stock of flows for which the structure of the KMS states is well elucidated. These papers often build on a rather small collection of papers where the general structure is investigated for a larger class of flows. One of these papers is the one by S. Neshveyev, \cite{N}, in which he investigates the KMS states for a natural and very general class of flows on the $C^*$-algebra of an \'etale groupoid. He obtains a description of the KMS states in terms of measures and measurable fields of traces on the $C^*$-algebras of the isotropy groups in the groupoid. The generality of the setting means that his approach covers a very large class of examples and the results provide a natural setup which can be employed widely. Nonetheless it is often not an easy task to obtain from the general description a satisfying understanding of the KMS states in specific cases. In particular, we have often found it difficult to understand the fields of tracial states which enter in general. This fact, and a very recent paper by D. Ursu, \cite{U}, triggered the present work. The problem which Ursu considers is that of determining the set of traces on a crossed product $A \rtimes G$ by a discrete group $G$ which extends a given $G$-invariant trace on a $C^{*}$-algebra $A$. This problem is analogues to that of determining, in the setting of Neshveyev, the measurable fields of traces associated to a given quasi-invariant measure on the unit space of the groupoid. As we will show in the present paper this is much more than an analogy; in fact, the ideas of Ursu work well to give an algebraic alternative to the measurable fields of traces in Neshveyev's setting. See Section \ref{Sec4}. Moreover, Ursu's approach can also be used in a setting which simultaneously generalizes both the groupoid setting of Neshveyev and the crossed product setting of Ursu to give a general description of the KMS states for a very large class of flows in terms of traces on a subalgebra of the fixed point algebra of the flow and a family of completely positive maps, called pseudoexpectations in \cite{U}. See Section \ref{section1}. Ursu's ideas work even better for crossed products by abelian groups where they can be combined with the spectral analysis developed for abelian groups of automorphisms by Arveson, Borchers and Connes. This leads to what we think is a very satisfying general description of the KMS states for a large class of flows on a crossed product by an abelian group. The analogue of the measurable fields from Neshveyev's setting is here played by a subgroup of the dual group, the Connes spectrum of an action arising naturally in the setting. See Section \ref{secabel}. We conclude the paper by giving an example which illustrates some of the results and relate to our work on KMS states for flows on the crossed product of a homeomorphism, \cite{CT}. The example demonstrates also that the main virtue of our results is that they set up a general framework which can be used in cases that have not been considered before, and that they leave plenty of hard work to be done in order to get the full picture of the structure of KMS states for specific flows. Just like the results of Neshveyev that motivated the work.

\smallskip

 \emph{Acknowledgement} The work was supported by the DFF-Research Project 2 `Automorphisms and Invariants of Operator Algebras', no. 7014-00145B.

\section{The setting} \label{section1}
Let $B$ be a $C^*$-algebra.
A flow on $B$ is a point-wise norm-continuous representation $ \alpha =\left\{\alpha_t\right\}_{t \in \mathbb R}$ of the real line $\mathbb R$ by automorphisms of ${B}$. Given such a flow we say that $b\in {B}$ is \emph{$\alpha$-analytic} when the map $t \to \alpha_t(b)$ can be extended to a holomorphic map $\mathbb C \ni z \mapsto \alpha_z(b)$. Let $\beta \in \mathbb R$. A state $\omega$ on $B$ is a \emph{$\beta$-KMS state} for $\alpha$ when
\begin{equation}\label{06-06-20a}
\omega(ba) = \omega(a\alpha_{i\beta}(b))
\end{equation}
for all $a \in B$ and all elements $b$ of a dense $\alpha$-invariant $*$-subalgebra of $\alpha$-analytic elements, and it holds then for all $\alpha$-analytic elements $b$, cf. \cite{BR}. Let $A \subseteq B$ be a $C^*$-subalgebra of $B$. We say that an element $b \in {B}$ is an \emph{$(A,\alpha)$-normalizer} when $b$ is $\alpha$-analytic and has the property that
\begin{equation}\label{06-06-20}
\alpha_{t}(b) {A} b^{*} \subseteq {A} \ \text{  and  } \  b^{*} {A} \alpha_{t}(b) \subseteq {A} \ \text{ for all } t \in \mathbb{R}.
\end{equation}
We denote the set of $(A,\alpha)$-normalisers by $N_\alpha({A})$. We say that ${A}$ is \emph{$\alpha$-regular} when 
\begin{enumerate}
\item ${A}$ contains an approximate unit for ${B}$, 
\item ${A}$ is contained in the fixed-point algebra of $\alpha$, i.e. $\alpha_{t}(a)=a$ for all $a\in {A}$ and all $t\in \mathbb{R}$, and
\item ${B}$ is generated as a $C^{*}$-algebra by $N_\alpha({A})$.
\end{enumerate}
Since $N_\alpha({A})$ is closed under composition and taking adjoints, the last condition holds if and only if
\begin{equation*}
{B}=\overline{\text{span}} \ N_\alpha({A})    \ .
\end{equation*}


\section{KMS states on $B$ and trace states on $A$}\label{Sec2}
Throughout this section $\alpha$ is a flow on $B$ and $A \subseteq B$ is an $\alpha$-regular subalgebra.

\begin{lemma}\label{06-06-20g} Let $a \in A, \ b \in N_\alpha(A)$. Then $b^*a\alpha_{i\beta}(b) \in A$.
\end{lemma}
\begin{proof} Assume for a contradiction that $b^{*}a\alpha_{i\beta}(b) \notin {A}$. By the Hahn-Banach theorem there is a continuous linear functional $\phi$ on ${B}$ with $\phi({A})=\{0\}$ and $\phi(b^{*}a\alpha_{i\beta}(b))= 1$. But $z\to \phi(b^{*}a\alpha_{z}(b))$ is holomorphic and $0$ on $\mathbb{R}$, which implies that $\phi(b^{*}a\alpha_{i\beta}(b))=0$.
\end{proof}

We let $T({A})$ denote the set of trace states on ${A}$. Let $\beta \in \mathbb R$. An element $\tau \in T({A})$ is \emph{$(\alpha,\beta)$-conformal} when
\begin{equation} \label{eg1}
\tau(bab^{*}c)=\tau(ab^{*}c\alpha_{i\beta}(b)) \text{ for all } a, c\in {A} \text{ and } b\in N_\alpha({A}) \ .
\end{equation}
The righthand side of \eqref{eg1} makes sense by Lemma \ref{06-06-20g}.

\begin{lemma} Let $\tau \in T(A)$. Then
$\tau$ is $(\alpha,\beta)$-conformal if and only if 
\begin{equation}\label{06-06-20d}
\tau(bb^{*})=\tau(b^{*}\alpha_{i\beta}(b))
\end{equation} 
for all $b\in N_\alpha({A})$.
\end{lemma}
\begin{proof} One direction follows immediate from \eqref{eg1} by taking $a$ and $c$ from an approximate unit for $B$. Assume instead that \eqref{06-06-20d} holds. Let $b\in N_\alpha({A})$ and let $a,c\in {A}$ be positive. Then $\sqrt{c}b\sqrt{a}\in N_\alpha({A})$ and hence by assumption
\begin{align*}
\tau(bab^{*}c) &= \tau( \sqrt{c}b\sqrt{a} (\sqrt{c}b\sqrt{a})^{*} )
=\tau(  (\sqrt{c}b\sqrt{a})^{*} \alpha_{i\beta}(\sqrt{c}b\sqrt{a}))\\
&=\tau(  (\sqrt{c}b\sqrt{a})^{*}\sqrt{c} \alpha_{i\beta}(b)\sqrt{a}) = \tau(  a b^{*}  c \alpha_{i\beta}(b)) \ .
\end{align*}
It follows by linearity that \eqref{eg1} holds.
\end{proof}

 We denote the set of $(\alpha,\beta)$-conformal trace states by $T_\beta(A)$ and we denote by $S_{\beta}({B})$ the set of $\beta$-KMS states for $\alpha$ on $B$. Clearly,
\begin{equation}\label{06-06-20e}
\omega \in S_{\beta}({B}) \ \Rightarrow \ \omega|_{{A}}\in T_\beta({A}) \ .
\end{equation}

Given a $C^*$-algebra $D$ and a state $\psi$ on $D$ we denote in the following by $\pi_{\psi}$ the representation of $D$ by bounded operators on the Hilbert space $H_{\psi}$ obtained from the GNS construction and by $\Omega_{\psi}\in H_{\psi}$ a cyclic vector such that $\psi(d) = \left<\Omega_{\psi},\pi_{\psi}(d)\Omega_{\psi}\right>$. We denote by $\overline{\psi}$ the normal state on $\pi_{\psi}(D)''$ given by $\Omega_{\psi}$, i.e. $\overline{\psi}(\cdot) = \left< \Omega_{\psi}, \cdot \  \Omega_{\psi}\right>$.

\begin{lemma}\label{lemma62}
Let $\tau \in T_\beta({A})$ and let $\omega$ be a state on ${B}$ with $\omega|_{A} = \tau$. There is an isomorphism $\iota_{\omega}  : \pi_\tau({A})'' \to \pi_{\omega}({A})''$ of von Neumann algebras such that $\iota_{\omega} \circ \pi_\tau = \pi_\omega|_{A}$ and $\overline{\tau} = \overline{\omega} \circ \iota_\omega$.
\end{lemma}
\begin{proof} We define an isometry $W : H_{\tau} \to H_{\omega}$ such that $W\pi_\tau(a)\Omega_{\tau} = \pi_{\omega}(a)\Omega_{\omega}$ for $a\in {A}$. Then $W\pi_\tau(a) = \pi_\omega(a) W$ and $WW^*$ is the projection onto the subspace $\overline{\pi_\omega({A})\Omega_{\omega}}$ of $H_{\omega}$. Since $WW^*$ is in the commutant of $\pi_\omega({A})''$ we can define a normal $*$-homomorphism $j : \pi_{\omega}({A})'' \to \pi_\tau({A})''$ such that $j(m) = W^*mW$. Note that $j \circ \pi_\omega|_{A} = \pi_\tau$ which in particular implies that $j$ is surjective. To see that $j$ is injective, assume that $m \in \pi_{\omega}({A})''$ and $j(m) = 0$. It follows that $mWW^* = 0$ and hence that $m\overline{\pi_\omega({A})\Omega_{\omega}}= \{0\}$. Now let $b\in N_\alpha({A})$ and let $\{a_{n}\}_{n=1}^{\infty} \subseteq {A}$ satisfy that $\pi_{\omega}(a_{n})\to m^{*}m \in \pi_{\omega}({A})''$ in the strong operator topology. Then
\begin{align*}
&\lVert m  \pi_{\omega}(b^{*}) \Omega_{\omega} \rVert^{2} =\overline{\omega}(\pi_{\omega}(b) m^{*} m \pi_{\omega}(b)^{*}) = \lim_{n\to \infty} \overline{\omega}(\pi_{\omega}(b) \pi_{\omega}(a_{n}) \pi_{\omega}(b)^{*}) \\
&= \lim_{n\to \infty} \omega(b  a_{n}b^{*})
= \lim_{n\to \infty} \tau(b  a_{n}b^{*}) = \lim_{n\to \infty} \tau( a_{n}b^{*}\alpha_{i\beta}(b)) \\ &
= \lim_{n\to \infty} \overline{\omega}(  \pi_{\omega}(a_{n})\pi_{\omega}(b^{*}\alpha_{i\beta}(b))) =\overline{\omega}(  m^{*}m\pi_{\omega}(b^{*}\alpha_{i\beta}(b))) \\
&= \langle \Omega_{\omega}, m^{*}m\pi_{\omega}(b^{*}\alpha_{i\beta}(b)) \Omega_{\omega}  \rangle=0 \ ,
\end{align*}
because $m\pi_{\omega}(b^{*}\alpha_{i\beta}(b)) \Omega_{\omega}=0$ since $b^{*}\alpha_{i\beta}(b) \in {A}$ by Lemma \ref{06-06-20g}. The vectors $\pi_{\omega}(b^{*})\Omega_{\omega}$ with $b\in N_\alpha({A})$ span a dense subspace of $H_{\omega}$ and we conclude that $m=0$. It follows that $j$ is also injective and hence an isomorphism. Set $\iota_{\omega} = j^{-1}$. The observation that $\overline{\tau} \circ j(m) = \left< \Omega_\tau, W^*mW\Omega_\tau\right> = \left< W\Omega_\tau, mW\Omega_\tau\right> = \left< \Omega_\omega, m\Omega_\omega\right> = \overline{\omega}(m)$ completes the proof.
\end{proof}

In the following we shall often suppress $\iota_\omega$ in notation and instead consider $\pi_\tau(A)''$ as a subalgebra of $\pi_\omega(B)''$.

\begin{lemma}\label{07-06-20}
Let $\tau \in T_\beta({A})$. For each $b\in N_\alpha({A})$ there is a bounded strong operator continuous linear map $\gamma_{b} : \pi_{\tau}({A})'' \to \pi_{\tau}({A})''$ such that
\begin{equation*}
\gamma_{b}(\pi_{\tau}(a))=\pi_{\tau}(bab^{*}) \ \text{ for all } a\in {A}.
\end{equation*}
The map $N_\alpha({A}) \ni b \to \gamma_{b}$ is a semi-group representation in the sense that $\gamma_{ab}=\gamma_{a}\gamma_{b}$.
\end{lemma}

\begin{proof}
Since  $\tau$ is a state on ${A}$ there is a state $\omega$ on ${B}$ such that $\omega|_{{A}}=\tau$. Let $\iota_\omega : \pi_\tau(A)'' \to \pi_\omega(A)''$  be the isomorphism from Lemma \ref{lemma62} and set
$$
\gamma_b(m)=  \iota_\omega^{-1}\left( \pi_\omega(b)\iota_\omega(m)\pi_\omega(b^*)\right) \ .
$$
We leave it to the reader to check that $\gamma$ has the properties stated.

\end{proof}



Notice that when $\tau \in T_\beta({A})$ it follows from \eqref{eg1} that
\begin{equation} \label{ekms}
\overline{\tau}(\gamma_{b}(a) \pi_\tau(c))=\overline{\tau}(a \pi_\tau(b^* c \alpha_{i\beta}(b)))
\end{equation} 
for all $a\in \pi_{\tau}({A})'', \ c \in A$ and $b\in N_\alpha({A})$.

\begin{lemma}\label{lemma64} Let $\tau \in T_\beta({A})$ and assume that there is an $\omega \in S_{\beta}({B})$ such that $\omega|_{A} = \tau$. There is a normal faithful conditional expectation $F'_{\omega} : \pi_{\omega}({B})'' \to \pi_\tau({A})''$ such that $\overline{\tau} \circ F'_{\omega} = \overline{\omega}$. 
\end{lemma}
\begin{proof} This follows from \cite{Ta}.
\end{proof}

In the setting of Lemma \ref{lemma64}, set 
\begin{equation}\label{econd}
F_{\omega} = F'_{\omega} \circ \pi_{\omega}: \ {B} \ \to \ \pi_{\tau}({A})'' \ . 
\end{equation}

\begin{lemma}  \label{lem17}Let $\tau \in T_\beta({A})$ and assume that $\omega \in S_{\beta}({B})$ satisfies $\omega|_{A} = \tau$. Let $F_{\omega}$ be the map defined in \eqref{econd}. Then

\begin{itemize}
\item[a)] $F_{\omega}$ is a completely positive contraction,
\item[b)] $F_{\omega}|_{A} = \pi_{\tau}$,
\item[c)] $\overline{\tau} \circ F_{\omega} = \omega $, and
\item[d)] $F_{\omega}(bxb^{*}) = \gamma_{b}\left(F_{\omega}(x)\right)$ for all $x\in {B}$ and all $b\in N_\alpha({A})$.
\end{itemize}
\end{lemma}
\begin{proof} a) and b) follow immediately from the constructions. c): Let $x \in B$. Then
$$
\overline{\tau}\left(F_{\omega}(x)\right)= \overline{\tau}\left(F'_{\omega}(\pi_{\omega}(x))\right) = \overline{\omega}\left(\pi_{\omega}(x)\right) = \omega(x) \ .
$$ 
d): Let $x \in B$ and $a\in A$, and notice that $b)$ implies that $A$ lies in the multiplicative domain of $F_{\omega}$. Since $\tau \in T_\beta({A})$ and $\omega \in S_{\beta}({B})$ we can use $c)$ and \eqref{ekms} to conclude that
\begin{align*}
&\overline{\tau}\left(\gamma_{b}(F_{\omega}(x)) \pi_{\tau}(a) \right)
   = \overline{\tau}\left(F_{\omega}(x) \pi_{\tau}(b^{*}a\alpha_{i\beta}(b)) \right)
   = \overline{\tau}\left( F_{\omega}(xb^{*}a\alpha_{i\beta}(b) )\right) \\ & = {\omega}\left(x b^{*}a \alpha_{i\beta}(b) \right) = \omega\left(bxb^{*}a \right)  = \overline{\tau} (F_{\omega}(bxb^{*})\pi_{\tau}(a)) \ .
\end{align*}
Since $\overline{\tau}$ is faithful we find that $F_{\omega}(bxb^{*})=\gamma_{b}(F_{\omega}(x))$.
\end{proof}

\begin{lemma} \label{thm67}
Let $\tau \in T_\beta({A})$. Let $F : {B} \to \pi_{\tau}({A})''$ satisfy that
\begin{itemize}
\item[a)] $F$ is a completely positive contraction,
\item[b)] $F|_{A} = \pi_{\tau}$, and
\item[c)] there is a subset $\mathcal F \subseteq N_\alpha({A})$ which spans a dense $\alpha$-invariant $*$-subalgebra of $B$ such that
$$
F(bxb^{*}) = \gamma_{b}\left(F(x)\right)
$$ 
for all $x\in {B}$ and all $b\in \mathcal{F}$. 
\end{itemize}
Then $\overline{\tau}\circ F \in S_{\beta}({B})$.
\end{lemma}

\begin{proof} 
Set $\varphi=\overline{\tau}\circ F$, which is clearly a state by a) and b). To prove the lemma it suffices to prove that $\varphi(bd)=\varphi(d\alpha_{i\beta}(b))$ when $b,d\in \mathcal{F}$. Now fix $b,d\in \mathcal{F}$. Notice that we can assume that $\lVert b \rVert =1$. Since ${A}$ contains an approximate unit for $B$, we have that $bb^{*}\in {A}$, which implies that $(bb^{*})^{1/n}\in {A}$, while $\pi_{\varphi}(bb^{*})^{1/n}$ converges to the range-projection $P$ of $\pi_{\varphi}(bb^{*})$ when $n \to \infty$, and hence $P\in \pi_{\varphi}({A})''$. However, the range projections of $\pi_{\varphi}(bb^{*})$ and $\pi_{\varphi}(b)$ coincide. Letting $\{a_{n}\}\subseteq {A}$ satisfy that $\pi_{\varphi}(a_{n})\to P$ in the strong operator topology, we get that
\begin{align*}
\varphi(bd)&=\overline{\varphi}(\pi_{\varphi}(b)\pi_{\varphi}(d))=
\overline{\varphi}(P\pi_{\varphi}(b)\pi_{\varphi}(d))
= \lim_{n\to \infty} \overline{\varphi}(\pi_{\varphi}(a_{n})\pi_{\varphi}(b)\pi_{\varphi}(d)) \\
&=\lim_{n\to \infty} \overline{\tau}\circ F (a_{n}bd)  \ .
\end{align*}
Condition $b)$ implies that $A$ is in the multiplicative domain of the completely positive contraction $F$. We can therefore continue the calculation as follows:
\begin{align*}
\varphi(bd)&= \lim_{n\to \infty} \overline{\tau}\left ( F (a_{n}) F(bd) \right) 
= \lim_{n\to \infty} \overline{\tau}\left (  F(bd) F (a_{n}) \right)= \lim_{n\to \infty} \overline{\tau}\circ F\left ( bda_{n} \right) \\
&= \lim_{n\to \infty} \overline{\varphi}(\pi_{\varphi}(b)\pi_{\varphi}(d)\pi_{\varphi}(a_{n}))  = \overline{\varphi}(\pi_{\varphi}(b)\pi_{\varphi}(d)P)\ .
\end{align*}
Now $P$ is the strong limit of a sequence $g_{n}(\pi_{\varphi}(bb^{*}))$ where $0\leq g_{n}\leq 1$ is an increasing sequence of continuous functions on $[0,\infty[$ such that $g_{n}$ is zero in a neighborhood of $0$ for all $n$ and $\lim_{n \to \infty} g_{n}(x) = 1$ for all $x>0$. It follows that we can find an increasing sequence of functions $\{f_{n}\}_{n=1}^{\infty} \subseteq C([0, \infty [)$ such that 
$$
f_{n}(\pi_{\varphi}(bb^{*}))\pi_{\varphi}(bb^{*}) \to P
$$
strongly as $n$ goes to $\infty$. Since $\overline{\varphi}$ is strongly continuous we get that
\begin{align*}
&\overline{\varphi}(\pi_{\varphi}(b)\pi_{\varphi}(d)P) =\lim_{n \to \infty}\overline{\varphi}(\pi_{\varphi}(b)\pi_{\varphi}(d)f_{n}(\pi_{\varphi}(bb^{*}))\pi_{\varphi}(bb^{*})) \\
&=\lim_{n  \to \infty}\overline{\tau}\left( F(bd f_{n}(bb^{*})bb^{*}) \right) =\lim_{n\to \infty}\overline{\tau}\left( \gamma_{b}(F(d f_{n}(bb^{*})b)) \right) \ .
\end{align*}
Let $\{u_k\}$ be an approximate unit for $B$ contained in $A$.
We can then apply \eqref{ekms} to find that
\begin{align*}
&\overline{\tau}\left( \gamma_{b}(F(d f_{n}(bb^{*})b)) \right) = \lim_{k \to \infty} \overline{\tau}\left( \gamma_{b}(F(d f_{n}(bb^{*})b)) \pi_\tau(u_k) \right) \\
& = \lim_{k \to \infty}\overline{\tau}\left( F(d f_{n}(bb^{*})b)\pi_{\tau}(b^{*}u_k\alpha_{i\beta}(b)) \right) = \overline{\tau}\left( F(d f_{n}(bb^{*})b)\pi_{\tau}(b^{*}\alpha_{i\beta}(b)) \right) \ .
\end{align*}
We can therefore continue the calculation to find that
\begin{align*}
\varphi(bd)&=\lim_{n\to \infty}\overline{\tau}\left( F(d f_{n}(bb^{*})b)\pi_{\tau}(b^{*}\alpha_{i\beta}(b)) \right) =\lim_{n\to \infty}\overline{\tau}\left( F(d f_{n}(bb^{*})bb^{*}\alpha_{i\beta}(b)) \right)\\
&=\lim_{n\to \infty}\overline{\varphi}
(\pi_{\varphi}(d)f_{n}(\pi_{\varphi}(bb^{*}))\pi_{\varphi}(bb^{*})\pi_{\varphi}(\alpha_{i\beta}(b))) \\
&=\overline{\varphi}\left( \pi_{\varphi}(d) P\pi_{\varphi}(\alpha_{i\beta}(b)) \right) \ .
\end{align*}
For the final step in the calculation observe that $\alpha_t(b)\alpha_t(b^*) = \alpha_t(bb^*) = bb^*$ which implies that $\pi_\varphi(\alpha_t(b))$ and $\pi_\varphi(b)$ have the same range projection. It follows that $P\pi_\varphi(\alpha_t(b)) - \pi_\varphi(\alpha_t(b)) = 0$ for all $t \in \mathbb R $, and since $z \mapsto P\pi_\varphi(\alpha_z(b)) - \pi_\varphi(\alpha_z(b))$ is entire it follows that it is constant $0$. In particular, $P\pi_\varphi(\alpha_{i\beta}(b)) = \pi_\varphi(\alpha_{i \beta}(b))$. We can therefore now complete the calculation to get that
$$
\varphi(bd)=\overline{\varphi}(\pi_{\varphi}(d) \pi_{\varphi}(\alpha_{i\beta}(b)) ) =\varphi(d\alpha_{i\beta}(b))  \ .
$$

\end{proof}

\begin{thm}\label{07-06-20d} Let $\tau \in T_\beta({A})$. The map
$$
F \mapsto \overline{\tau}\circ F
$$
is a bijection between the set of maps $F : B \to \pi_\tau(A)''$ satisfying conditions a), b) and c) in Lemma \ref{thm67} and the set $\left\{ \omega \in S_{\beta}({B}) \ | \ \omega|_{{A}}=\tau  \right\}$. 
\end{thm}

\begin{proof}
Lemma \ref{thm67} and Lemma \ref{lem17} imply that the map is well-defined and surjective.  To see that it is injective, assume that $F_{1}$ and $F_{2}$ are two maps satisfying conditions a), b) and c) in Lemma \ref{thm67}, and that $\overline{\tau}\circ F_{1} = \overline{\tau}\circ F_{2}$. Let $x\in {B}$ and $a\in {A}$. Then
\begin{equation*}
0=\overline{\tau}\circ F_{1}(ax)-\overline{\tau}\circ F_{2}(ax)
=\overline{\tau}(\pi_{\tau}(a)(F_{1}(x)-F_{2}(x))) \ .
\end{equation*}
Since $\overline{\tau}$ is strongly continuous and faithful we conclude that $F_1 = F_2$.
\end{proof}

In particular, it follows that if $F$ is a map satisfying conditions a), b) and c) in Lemma \ref{thm67} then $F(bxb^*) = \gamma_b(F(x))$ for all $x \in B$ and all $b \in N_\alpha(A)$.

\section{Groupoid $C^*$-algebras: Relation to Neshveyev's theorem}\label{Sec4}

In this section we show in which way Theorem \ref{07-06-20d} can be considered as a generalization of Neshveyev's result from \cite{N} where he extends Renaults work on KMS states for flows on groupoid $C^*$-algebras from the case of principal groupoids handled in \cite{Re} to general \'etale groupoids. We refer to \cite{Re} for an introduction to groupoid $C^*$-algebras and to \cite{N} for the terminology and notation we shall use here.

Let $\mathcal{G}$ be a locally compact second countable Hausdorff \'etale groupoid. The maps $r(g)=gg^{-1}$ and $s(g)=g^{-1}g$ on $\mathcal{G}$ are then local homeomorphims, and their range $\mathcal{G}^{(0)}$ is the unit space. We set $\mathcal{G}^{x}=r^{-1}(x)$, $\mathcal{G}_{x}=s^{-1}(x)$ and $\mathcal{G}^{x}_{x}=\mathcal{G}_{x}\cap \mathcal{G}^{x}$ for $x\in \mathcal{G}^{(0)}$, and recall that on the dense $*$-subalgebra $C_{c}(\mathcal{G})$ of the (full) groupoid $C^{*}$-algebra $C^{*}(\mathcal{G})$ the product is given by
\begin{equation}\label{eGprod}
(f_{1}*f_{2})(g)=\sum_{h\in \mathcal{G}^{r(g)}} f_{1}(h) f_{2}(h^{-1}g)
\end{equation}
and the involution by $f^{*}(g)=\overline{f(g^{-1})}$.

Let $\mu$ be a Borel probability measure on the unit space $\mathcal{G}^{(0)}$ of $\mathcal G$. The fundamental notion in \cite{N} is the following: A collection of states $\{\varphi_{x}\}_{x\in \mathcal{G}^{(0)}}$ with $\varphi_{x}$ a state on $C^{*}(\mathcal{G}_{x}^{x})$ is a \emph{$\mu$-measurable field of states} when the map
$$
\mathcal{G}^{(0)} \ni x \to \sum_{g\in \mathcal{G}_{x}^{x}} f(g)\varphi_{x}(u_{g})
$$
is $\mu$-measurable for all $f\in C_{c}(\mathcal{G})$, where $u_g, \ g \in \mathcal G_x^x$, denotes the canonical unitaries in $C^*(\mathcal G_x^x)$. Two $\mu$-measurable fields $\{\varphi_{x}\}_{x\in \mathcal{G}^{(0)}}$ and $\{\psi_{x}\}_{x\in \mathcal{G}^{(0)}}$ are identified if $\psi_{x}=\varphi_{x}$ for $\mu$-almost all $x\in \mathcal{G}^{(0)}$.

\begin{lemma} \label{prop1}
The following two sets are in bijective correspondence:
\begin{enumerate}
\item The set of $\mu$-measurable fields of states, and
\item the set of contractive completely positive linear maps $E:C^{*}(\mathcal{G})\to \pi_{\mu}(C_{0}(\mathcal{G}^{(0)}))''$ with $E(f)=\pi_{\mu}(f)$ for all $f\in C_{0}(\mathcal{G}^{(0)})$.
\end{enumerate}
\end{lemma} 

\begin{proof} To define the map from $(1)$ to $(2)$, let $\pi:C_{0}(\mathcal{G}^{(0)}) \to B(L^{2}(\mathcal{G}^{(0)}, \mu))$ be the representation by multiplication operators. Since $(L^{2}(\mathcal{G}^{(0)}, \mu), \pi, 1)$ is a GNS triple for $\mu$ we identify $M:=L^{\infty}(\mathcal{G}^{(0)}, \mu)$ with $ \pi_{\mu}(C_{0}(\mathcal{G}^{(0)}))''$. Let $\{\varphi_{x}\}_{x\in \mathcal{G}^{(0)}}$ be a $\mu$-measurable field of states. By (the proof of) Theorem 1.1 in \cite{N} the map
$$
C_{c}(\mathcal{G}) \ni f \to \sum_{g\in \mathcal{G}_{x}^{x}} f(g) \varphi_{x}(u_{g})
$$
extends to a state on $C^{*}(\mathcal{G})$. We can therefore define a positive linear map $E:C_{c}(\mathcal{G})\to M$ such that 
$$
E(f)(x)=\sum_{g\in \mathcal{G}_{x}^{x}} f(g) \varphi_{x}(u_{g}) \text{ for } x\in \mathcal{G}^{(0)} \ .
$$
Since $\lVert E(f) \rVert_{\infty} \leq \lVert f \rVert$ we can extend $E$ to a positive linear contractive map $E:C^{*}(\mathcal{G})\to M$. Since $M$ is abelian $E$ is actually completely positive, cf. e.g. A.3 on page 266 in \cite{NS}. For $f\in C_{0}(\mathcal{G}^{(0)})$ we have that $\sum_{g\in \mathcal{G}_{x}^{x}} f(g) \varphi_{x}(u_{g})=f(x)$ and hence $E(f)=\pi_{\mu}(f)$. Thus we have constructed a map from (1) to (2). Let $\overline{\mu}$ be the normal state on $M$ extending $\mu$. Then
$$
\overline{\mu}\circ E(f)  = \int_{\mathcal G^{(0)}} \sum_{g\in \mathcal{G}_{x}^{x}} f(g) \varphi_{x}(u_{g})  \ \mathrm{d}\mu(x)
$$
when $f \in C_c(\mathcal G)$. It follows therefore from Theorem 1.1 in \cite{N} that the map from (1) to (2) is injective. To see that it is surjective, let $F:C^{*}(\mathcal{G})\to M$ be a map satisfying the conditions in $(2)$. Then $\overline{\mu} \circ F$ is a state on $C^*(\mathcal G)$ and since $C_0(\mathcal G^{(0)})$ is in the multiplicative domain of $F$ we find that
$$
\overline{\mu} \circ F(fa) = \overline{\mu}(\pi_\mu(f)F(a)) = \overline{\mu}(F(a)\pi_\mu(f)) = \overline{\mu} \circ F(af) 
$$
for all $f \in C_0(\mathcal G^{(0)})$ and all $a \in C^*(\mathcal G)$. Since $\overline{\mu} \circ F$ restricts to $\mu$ on $C_0(\mathcal G^{(0)})$ it follows from Theorem 1.1 in \cite{N} that $\overline{\mu}\circ F$ is given by a $\mu$-measurable field $\{\varphi_{x}\}_{x\in \mathcal{G}^{(0)}}$. Let $E:C^{*}(\mathcal{G})\to M$ be the map obtained from this field as in the first part of the proof. For $f\in C_{c}(\mathcal{G})$ and $h\in C_{c}(\mathcal{G}^{(0)})$ we have that
\begin{align*}
&\int_{\mathcal{G}^{(0)}} h E(f) \ d \mu =  \int_{\mathcal{G}^{(0)}}  E(h f) \ d \mu = \overline{\mu}\circ E(hf) \\
&= \int_{\mathcal G^{(0)}} \sum_{g\in \mathcal{G}_{x}^{x}} (hf)(g) \varphi_{x}(u_{g})  \ \mathrm{d}\mu(x) = \overline{\mu} \circ F(hf) =\int_{\mathcal{G}^{(0)}} h F(f) \ d \mu 
\end{align*}
and hence $E(f)=F(f)$ $\mu$-almost everywhere, i.e. $E = F$.
\end{proof}

\begin{lemma}\label{08-06-20}
Let $\{\varphi_{x}\}_{x\in \mathcal{G}^{(0)}}$ be a $\mu$-measurable field of states for a quasi-invariant probability measure $\mu$ and let $E: C^{*}(\mathcal{G})\to \pi_{\mu}(C_{0}(\mathcal{G}^{(0)}))''$ be the map associated to $\{\varphi_{x}\}_{x\in \mathcal{G}^{(0)}}$ as in Lemma \ref{prop1}. The following two conditions are equivalent:
\begin{enumerate}
\item $\varphi_{x}(u_{g})=\varphi_{r(h)}(u_{hgh^{-1}})$ for $\mu$-a.e. $x\in \mathcal{G}^{(0)}$ and all $g\in \mathcal{G}_{x}^{x}$ and $h\in \mathcal{G}_{x}$,
\item $E(f^{*}af)=f^{*}E(a)f$ for all $a\in C^{*}(\mathcal{G})$ and all $f\in C_{c}(\mathcal{G})$ supported in a bisection.
\end{enumerate}
\end{lemma}

\begin{proof}
Assume that $k\in C_{c}(\mathcal{G})$ and $f$ is supported in a bisection $W$. Fix $x\in \mathcal{G}^{(0)}$ and a $g\in \mathcal{G}_{x}^{x}$. Using the definition of the product \eqref{eGprod} one can show that
\begin{equation}\label{eprod}
(f^{*}kf)(g) = 
\begin{cases}
|f(h)|^{2} k(hg h^{-1}) &  \text{ when } x=s(h) \text{ for a } h\in W \\
0 & \text{ for } x\notin s(W) \ .
\end{cases}
\end{equation}
If $x=s(h)$ for a $h\in W$ this implies that
\begin{align} \label{eprod2}
E(f^{*}kf)(x)&=\sum_{g\in\mathcal{G}_{x}^{x}} (f^{*}kf)(g) \varphi_{x}(u_{g}) \nonumber \\
&=|f(h)|^{2}\sum_{g\in\mathcal{G}_{x}^{x}} k(hg h^{-1}) \varphi_{x}(u_{g})   ,
\end{align}
while $E(f^{*}kf)(x)=0$ for $x\notin s(W)$. Since $f^{*}E(k)f$ is supported on $\mathcal{G}^{(0)}$ we can use \eqref{eprod} to get
\begin{equation*}
(f^{*}E(k)f)(x) = 
\begin{cases}
|f(h)|^{2} E(k)(h h^{-1}) &  \text{ when } x=s(h) \text{ for a } h\in W \\
0 & \text{ for } x \notin s(W) \ .
\end{cases}
\end{equation*} 
When $x=s(h)$ for a $h\in W$ this implies that
\begin{align} \label{eprod3}
(f^{*}E(k)f)(x)&=|f(h)|^{2} E(k)(r(h)) \nonumber \\
&= |f(h)|^{2}\sum_{g\in\mathcal{G}_{r(h)}^{r(h)}} k(g ) \varphi_{r(h)}(u_{g})  \ .
\end{align}
If we assume that $(1)$ is true, we get for $\mu$-almost all $x\in s(W)$ that
$$
\sum_{g\in\mathcal{G}_{x}^{x}} k(hg h^{-1}) \varphi_{x}(u_{g})
=\sum_{g\in\mathcal{G}_{x}^{x}} k(hg h^{-1}) \varphi_{r(h)}(u_{hgh^{-1}})
=\sum_{g\in\mathcal{G}_{r(h)}^{r(h)}} k(g ) \varphi_{r(h)}(u_{g})
$$
and comparing \eqref{eprod2} and \eqref{eprod3} will then imply that $(2)$ is true. If $(2)$ is true, we can use that $k \in C_c(\mathcal G)$ and $f$ are arbitrary to deduce from \eqref{eprod2} and \eqref{eprod3} that
$\varphi_{x}(u_{g})=\varphi_{r(h)}(u_{hgh^{-1}})$ for $\mu$-a.e. $x\in s(W)$ and all $g\in \mathcal{G}_{x}^{x}$ and $h\in W$. Finally, since $\mathcal G$ can be covered by a countable collection of bisections $W$ we conclude that $(1)$ holds.
\end{proof}

Let $c : \mathcal G \to \mathbb R$ be a continuous groupoid homomorphism, called an $\mathbb R$-valued $1$-cocycle in \cite{N}. The associated flow $\sigma^c$ is defined such that
$$
\sigma^c_t(f)(g) = e^{itc(g)}f(g)
$$
when $f \in C_c(\mathcal{G})$. Since any element $f \in C_c(\mathcal G)$ which is supported in a bisection is an $(C_0(\mathcal G^{(0)}),\sigma^c)$-normalizer it is clear that $C_0(\mathcal G^{(0)})$ is $\sigma^c$-regular and hence Theorem \ref{thm67} applies. Using the lemmas above we obtain in this way Theorem \ref{thm55}, which is an extension of Theorem 1.3 in \cite{N}. Notice that condition (iii) in Theorem 1.3 in \cite{N} does not apply in our setting since we have a different definition of $0$-KMS states; we do not require them to be $\sigma^c$-invariant.

\begin{thm} \label{thm55} Let $\beta \in \mathbb R$. Let $\mu$ be a quasi-invariant Borel probability measure on $\mathcal{G}^{(0)}$ with Radon Nikodym cocycle $e^{-\beta c}$. The following three sets are in bijective correspondence
\begin{enumerate}
\item The $\beta$-KMS states that restricts to $\mu$ on $C_{0}(\mathcal{G}^{(0)})$
\item The contractive completely positive  maps 
$$
E:C^{*}(\mathcal{G})\to \pi_{\mu}(C_{0}(\mathcal{G}^{(0)}))''
$$ 
satisfying that $E|_{C_{0}(\mathcal{G}^{(0)})}=\pi_{\mu}$ and that $E(f^{*}af)=f^{*}E(a)f$ for all $a\in C^{*}(\mathcal{G})$ and $f \in C_c(\mathcal G)$ supported in a bisection.
\item The $\mu$-measurable fields of tracial states $\{\varphi_{x}\}_{x\in \mathcal{G}^{(0)}}$ that satisfies $\varphi_{x}(u_{g})=\varphi_{r(h)}(u_{hgh^{-1}})$ for $\mu$-a.e. $x\in \mathcal{G}^{(0)}$, all $g\in \mathcal{G}_{x}^{x}$ and $h\in \mathcal{G}_{x}$.
\end{enumerate}
The map from $(3)$ to $(2)$ is given by sending $\{\varphi_{x}\}_{x\in \mathcal{G}^{(0)}}$ to $E$, where 
$$
E(f)(x)=\sum_{g\in \mathcal{G}_{x}^{x}} f(g) \varphi_{x}(u_{g}) \text{ for } f\in C_{c}(\mathcal{G})  \ ,
$$
and the map from $(2)$ to $(1)$ is given by sending $E$ to the state
$$
C^{*}(\mathcal{G}) \ni a \mapsto \int_{\mathcal{G}^{(0)}} E(a) \ d \mu  \ .
$$

\end{thm}

\section{Crossed products}
Let $A$ be a separable unital $C^*$-algebra and $G$ a countable discrete group. Let $\gamma : G \to \Aut(A)$ be an action of $G$ on $A$ by automorphisms. We can then consider both the corresponding reduced and full crossed product $C^*$-algebra. 
Both $C^*$-algebras are generated by a copy of $A$ and the elements $u_g,\ g \in G$, of a unitary representation $u$ of $G$ with the property that $u_gau_g^* = \gamma_g(a)$ for all $a \in A$ and all $g \in G$. It will not be necessary to distinguish between the reduced and the full crossed product and we let $A \rtimes_{\gamma} G$ denote either.

Let $G \ni g \mapsto D_g$ be a map (or cocycle) from $G$ into the set of self-adjoint elements in the center of $A$ such that
\begin{equation}\label{05-05-20}
\gamma_{h^{-1}}(D_g) \ + \ D_h \ = \ D_{gh} \ \ \ \forall g,h \in G\ .
\end{equation}
Such a cocycle $D$ defines a flow $\alpha^D$ on $A\rtimes_{\gamma} G$ determined by the condition that the formula
\begin{equation}\label{12-06-20b}
\alpha^D_t(au_g) = au_ge^{-it D_g} 
\end{equation}
holds for all $a\in A, \ g \in G$ and $t \in \mathbb R$. It is easy to see that $A$ is an $\alpha^D$-regular subalgebra of $A \rtimes_\gamma G$, and we can therefore apply the results of Section \ref{Sec2}.

Let $\beta \in \mathbb R$. Since the elements $a \in A$ and $u_g, g \in G$, span a dense $\alpha^{D}$-invariant $*$-subalgebra of $\alpha^D$-analytic elements it follows from Definition 5.3.1 in \cite{BR} that a state $\omega$ on $A \rtimes_{\gamma} G$ is a $\beta$-KMS state for $\alpha^D$ iff
\begin{equation}\label{15-05-20a}
\omega(au_gbu_h) = \omega\left(bu_h au_g e^{\beta D_g}\right) \ 
\end{equation}
for all $a,b \in A$ and all $g,h \in G$. Let $P : A \times_{\gamma} G \to A$ be the canonical conditional expectation.

\begin{lemma}\label{05-05-20b} Let $\beta \in \mathbb R$. 
\begin{enumerate}
\item[a)] Let $\omega$ be a $\beta$-KMS state for $\alpha^D$. Then 
\begin{itemize} 
\item $\omega|_A$ is a trace state on $A$, and
\item $\omega(\gamma_g(a)) \ = \ \omega\left(a e^{\beta D_g}\right) \ \ \ \forall a \in A \ \forall g \in G$.
\end{itemize}
\item[b)] Let $\tau$ be a trace state on $A$. Then $\tau \circ P$ is a $\beta$-KMS state for $\alpha^D$ iff  $\tau(\gamma_g(a)) \ = \ \tau\left(a e^{\beta D_g}\right) \  \forall a \in A \ \forall g \in G$. 
\end{enumerate} 
\end{lemma}
\begin{proof} a): Both items follow from \eqref{15-05-20a}; the first by taking $g = h = e$ (the neutral element in $G$) and the second by taking $a = 1$ and $h=g^{-1}$.

b): The necessity of the condition follows from a). To prove that it is also sufficient we calculate
$$
\tau \circ P(au_gbu_h) = \ \begin{cases} 0 \ , & \ g \neq h^{-1} \\ \tau\left(a\gamma_g(b)\right) \ , & \ g = h^{-1} \  \end{cases}
$$
and
$$
\tau \circ P\left(bu_hau_ge^{\beta D_g}\right) = \ \begin{cases} 0 \ , & \ g \neq h^{-1} \\ \tau\left(b\gamma_g^{-1}(a)e^{\beta D_g}\right) \ , & \ g = h^{-1} \ , \end{cases}
$$
which is the same because the assumed properties of $\tau$ imply that
$$
\tau\left(b\gamma_g^{-1}(a)e^{\beta D_g}\right) = \tau(\gamma_g(b)a) = \tau(a \gamma_g(b)) \ .
$$
\end{proof}

In the terminology of Section \ref{Sec2} we see that $\tau \in T(A)$ is $(\alpha^D,\beta)$-conformal if and only if it satisfies the condition in b) of Lemma \ref{05-05-20b}. It is therefore consistent that we denote by $T_\beta(A)$ the set of trace states on $A$ with the additional property specified in b) of Lemma \ref{05-05-20b} and let $S_{\beta}(A\rtimes_{\gamma} G)$ denote the set of $\beta$-KMS states for $\alpha^D$. Both are metrizable weak* compact convex sets; $S_{\beta}(A\rtimes_{\gamma} G)$ is even a Choquet simplex, cf. Theorem 5.3.30 in \cite{BR}. Lemma \ref{05-05-20b} has the following immediate corollary.

\begin{cor}\label{05-05-20c} The restriction map $\omega \mapsto \omega|_A$ is a continuous affine surjection of $S_{\beta}(A\rtimes_{\gamma} G)$ onto $T_\beta(A)$.
\end{cor}

We will mainly be concerned with the case where $G$ is abelian but before making that assumption we collect in this section some general conclusions.

\begin{lemma}\label{15-05-20c} Let $\tau \in T_\beta(A)$. There is a unitary representation $W$ of $G$ on $H_{\tau}$ such that $\Ad W_g \circ \pi_{\tau} = \pi_{\tau} \circ \gamma_g$.
\end{lemma}
\begin{proof} Since 
$$
\tau\left(\left(\gamma_g(a)e^{\frac{\beta}{2} D_{g^{-1}}}\right)^*\left( \gamma_g(a)e^{\frac{\beta}{2} D_{g^{-1}}}\right)\right) = \tau\left(\gamma_g(a^*a)e^{\beta D_{g^{-1}}}\right) =\tau(a^*a)$$
we can define a unitary $W_g$ on $H_{\tau}$ such that 
$$
W_g\pi_{\tau}(a)\Omega_{\tau} = \pi_{\tau}\left(\gamma_g(a)e^{\frac{\beta}{2} D_{g^{-1}}}\right)\Omega_{\tau} \ .
$$
It follows from \eqref{05-05-20} in a straightforward way that $W$ is a representation of $G$ with the stated property.
\end{proof}

In the following we denote by $\overline{\gamma}$ the action $\overline{\gamma}_g = \Ad W_g$ of $G$ by automorphisms of $\pi_{\tau}(A)''$. Notice that $\overline{\gamma}_g$ corresponds to $\gamma_{u_{g}}$ for $g\in G$ in the terminology of Section \ref{Sec2}. It is straightforward to obtain the following from Theorem \ref{thm67}; it generalizes a part of Theorem 1.2 in \cite{U}.
\begin{thm}\label{08-06-20b} Let $\tau \in T_\beta({A})$. The map
$$
F \mapsto \overline{\tau}\circ F
$$
is a bijection between 
$$
\left\{ \omega \in S_{\beta}(A\rtimes_{\gamma} G): \ \omega|_A = \tau \right\}
$$ 
and the set of unital completely positive maps $F : A \rtimes_\gamma G  \to \pi_\tau(A)''$ with the properties that
\begin{itemize}
\item $F|_A = \pi_\tau$, and
\item $F(u_gxu_g^*) = \overline{\gamma}_g(F(x))$ for all $g \in G$ and all $x \in A \rtimes_\gamma G$.
\end{itemize} 
\end{thm}

It is possible, exactly as in \cite{U}, to consider instead of $F$ itself only the elements $x_g = F(u_g)$. We shall not pursue this observation here, except when $G$ is abelian where it is most fruitful. See Lemma \ref{18-05-20}.

Let $\partial_e S_{\beta}(A \rtimes_{\gamma} G)$ and $\partial_e T_\beta(A)$ denote the sets of extremal elements in $S_{\beta}(A \rtimes_\gamma G)$ and $T_\beta(A)$, respectively.
The following generalizes Lemma 4.1 in \cite{Th}.

\begin{lemma}\label{13-05-20} Let $\omega \in \partial_e S_{\beta}(A \rtimes_{\gamma} G)$. Then 
\begin{itemize}
\item $\omega|_A \in \partial_e T_\beta(A)$.
\item For every $\tau \in \partial_e T_\beta(A)$ there is an $\omega \in \partial_e S_\beta(A \rtimes_\gamma G)$ such that $\omega|_A = \tau$.
\end{itemize}
\end{lemma}
\begin{proof} For the first item let $t \in ]0,1[, \ \tau_1,\tau_2 \in T_\beta(A)$ satisfy that $\omega|_A = t\tau_1 + (1-t)\tau_2$. To show that $\tau_1 = \tau_2 = \omega|_A$, let $(H_{\omega},\pi_{\omega}, \Omega_{\omega})$ be the GNS-representation of $\omega$. Set $H^0_{\omega} = \overline{\pi_{\omega}(A)\Omega_{\omega}}$. Then $(H^0_{\omega}, \pi_{\omega}|_A,\Omega_{\omega})$ is the GNS representation of $\omega|_A$, and we can identify $\pi_{\omega|_{A}}(A)''$ with $\pi_{\omega}(A)''$ by Lemma \ref{lemma62}. It follows then from Theorem 2.3.19 in \cite{BR} that there are positive operators $z_i \in \pi_{\omega}(A)'$ such that $\tau_i(\cdot) = \left<\Omega_{\omega}, z_i\pi_{\omega}(\cdot) \Omega_{\omega}\right>, \ i =1,2$. In particular, $\tau_i$ extends to a tracial state $\overline{\tau_i}$ on $\pi_{\omega}(A)''$ 
and Proposition 5.3.29 in \cite{BR} implies that $z_i $ is in the center of $\pi_{\omega}(A)''$. Now we use that $\overline{\omega}$ is a $\beta$-KMS state for the weak* continuous extension of $\alpha^D$ to $\pi_{\omega}(A \rtimes_{\gamma} G)''$, cf. Corollary 5.3.4 in [BR]. This yields
\begin{align*}
& \tau_i(\gamma_g(a)) = \tau_i(ae^{\beta D_g}) \ = \ 
 \left< \Omega_{\omega}, z_i \pi_{\omega}(a e^{\beta D_g})\Omega_{\omega}\right> \\ 
 & =  \left< \Omega_{\omega}, \pi_{\omega}(u_g) z_i \pi_{\omega}(a)\pi_{\omega}(u_g^*)\Omega_{\omega}\right>   =  \left< \Omega_{\omega}, \pi_{\omega}(u_g) z_i\pi_{\omega}(u_g^*) \pi_{\omega}(\gamma_g(a))\Omega_{\omega}\right> 
\end{align*}
for all $g\in G$ and $a\in A$. The uniqueness of $z_i$ now implies that $\pi_{\omega}(u_g) z_i\pi_{\omega}(u_g^*) = z_i$ for all $g \in G$. It follows that $z_i$ is actually in the center of $\pi_{\omega}(A \rtimes_{\gamma} G)''$, which is a factor since $\omega$ is extremal, cf. Theorem 5.3.30 in \cite{BR}. Thus the $z_i$'s are scalars, and we conclude that $\tau_1 = \tau_2 = \omega|_A$.

For the second item note that $\left\{ \omega \in S_{\beta}(A \rtimes_\gamma G) : \ \omega|_A = \tau \right\}$ is a closed face in $S_{\beta}(A \rtimes_\gamma G)$ since $\tau$ is extremal, and it is non-empty by Corollary \ref{05-05-20c}. Since it is a compact convex set it contains an extremal point by the Krein-Milman theorem and since it is a face such a point must be extremal in $S_\beta(A\rtimes_\gamma G)$ as well.

\end{proof}

\begin{lemma}\label{18-05-20a} Let $\tau \in \partial_e T_\beta(A)$. The action $\overline{\gamma}$ of $G$ on the center of $\pi_{\tau}(A)''$ is ergodic.
\end{lemma}
\begin{proof}
Let $0 \leq z \leq 1$ be a central element of $\pi_{\tau}(A)''$ fixed by $\overline{\gamma}$ and assume $z \neq 0$. Define 
$\tau_z (a) = \left<\Omega_{\tau}, z\pi_{\tau}(a) \Omega_{\tau} \right> $ for $a\in A$. Then $\tau_z$ is clearly tracial and we find that
\begin{align*}
&\tau_z(\gamma_g(a)) = \left<\Omega_{\tau}, z\overline{\gamma}_g(\pi_{\tau}(a)) \Omega_{\tau} \right> \\
& = \overline{\tau}\left(\overline{\gamma}_g(z\pi_{\tau}(a))\right) = \overline{\tau}\left(z\pi_{\tau}(ae^{\beta D_g})\right) \ = \tau_z(ae^{\beta D_g}) \ .
\end{align*} 
It follows $\overline{\tau}(z)^{-1}\tau_z \in T_\beta(A)$. Since $\tau_z \leq  \tau$ and $\tau$ is extremal, it follows that $\tau_z$ must be a scalar multiple of $\tau$; i.e. there is a $\lambda \geq 0$ such that $ \left<\Omega_{\tau}, (z -\lambda)\pi_{\tau}(a) \Omega_{\tau} \right> = 0$ for all $a \in A$. This implies that $z = \lambda$.  
\end{proof}

\section{Crossed products by abelian groups}\label{secabel}

In this section we retain the setting of the previous section but assume throughout that $G$ is abelian.

\begin{lemma}\label{18-05-20}  Let $\tau \in \partial_e T_\beta(A)$ and $\omega \in \partial_e S_{\beta}(A \rtimes_{\gamma} G)$ such that $\omega|_A = \tau$.  Let $\Gamma(\tau)$ be the Connes spectrum of the action $\overline{\gamma}$ on $\pi_{\tau}(A)''$, cf. \cite{Pe}. It follows that there is a representation $w$ of $\Gamma(\tau)^\perp$ by unitaries in ${\pi_{\tau}(A)''}$ such that 
\begin{itemize}
\item[i)] $\overline{\gamma}_g(w_h) = w_h$,
\item[ii)] $\overline{\gamma}_h = \Ad w_h$, 
\item[iii)] $\omega(au_g)  = \begin{cases} 0 \ , & g \notin \Gamma(\tau)^{\perp} \\ \overline{\tau}\left(\pi_\tau(a)w_g\right), & \ g \in \Gamma(\tau)^{\perp} \  \end{cases} $
\end{itemize}
for all $h \in \Gamma(\tau)^\perp$, all $g \in G$ and all $a \in A$.
\end{lemma}
\begin{proof} Let $F'_{\omega}$ be the conditional expectation from Lemma \ref{lemma64} and set $x_g = F'_{\omega}\circ \pi_{\omega}(u_g)$. For $a\in A$ then
\begin{equation}\label{26-05-20a}
\begin{split}
&x_g\pi_\tau(a) = F'_{\omega}(\pi_\omega(u_g)\pi_{\omega}(a)) \\
&=  F'_{\omega}(\pi_{\omega}(\gamma_g(a))\pi_\omega(u_g)) = \pi_{\tau}(\gamma_g(a))x_g \ .
\end{split}
\end{equation}
It follows that $x_gx_g^*\pi_\tau(a) = \pi_\tau(a)x_gx_g^*$ and in the same way also that $x_g^*x_g\pi_\tau(a) = \pi_\tau(a)x_g^*x_g$ for all $a \in A$. Thus $x_gx_g^*$ and $x_g^*x_g$ are both central elements of $\pi_\tau(A)''$. Since $G$ is abelian it follows from d) of Lemma \ref{lem17} that they are also both fixed by $\overline{\gamma}$. By Lemma \ref{18-05-20a} they are both non-negative scalars and hence $x_g$ is either $0$ or has the form $x_g = \lambda_g w_g$ for some scalar $\lambda_g > 0$ and a unitary $w_g$. If the latter occurs it follows from property d) in Lemma \ref{lem17} that $w_g$ is fixed by $\overline{\gamma}$ and by \eqref{26-05-20a} that $\overline{\gamma}_g = \Ad w_g$. By  Theorem 8.9.4 in \cite{Pe} this implies that $g$ is in the annihilator of the Borchers spectrum of $\overline{\gamma}$, but the Borchers spectrum is the same as the Connes spectrum in this case since the central support of any projection fixed by $\overline{\gamma}$ is $1$ by Lemma \ref{18-05-20a}. Thus 
$$
x_g \neq 0 \ \Rightarrow \ g \in \Gamma(\tau)^{\perp} \ .
$$
On the other hand, since the Arveson spectrum $\Sp(\overline{\gamma})$ is a closed subset of $\widehat{G}$, which is compact, it follows also from Theorem 8.9.4 in \cite{Pe} that $\overline{\gamma}_g = \Ad w'_g$ for some unitary $w'_g$ in the fixed point algebra of $\overline{\gamma}$ in $\pi_{\tau}(A)''$ when $g \in \Gamma(\tau)^{\perp}$. Then $w'_g\pi_{\omega}(u_g)^*$ is in the commutant of $\pi_{\tau}(A)''$ in $\pi_{\omega}(A \rtimes_\gamma G)''$ and commutes with $\pi_{\omega}(u_h)$ for all $h \in G$. Hence $w'_g\pi_{\omega}(u_g)^*$ is in the center of $\pi_{\omega}(A \rtimes_\gamma G)''$ which is a factor since $\omega$ is extremal, cf. Theorem 5.3.30 in \cite{BR}. It follows that $w'_g = {\lambda_g} \pi_{\omega}(u_g)$ for some $\lambda_g \in \mathbb C, \ |\lambda_g|= 1$, and hence that $x_g = F_{\omega}'(\pi_{\omega}(u_g)) = \overline{\lambda_g} w'_g = \pi_{\omega}(u_g)$ when $g \in \Gamma(\tau)^{\perp}$. When we set $w_h = x_h, \ h \in 
\Gamma(\tau)^{\perp}$, it follows that $w$ is a representation of $\Gamma(\tau)^\perp$ by unitaries in ${\pi_{\tau}(A)''}$ such that i) and ii) hold. We check that iii) holds:
\begin{equation*}
\omega(au_g) =  \overline{\tau}\left(F_{\omega}(au_g)\right) = \overline{\tau}(\pi_\tau(a)x_g) =  \begin{cases} 0 \ , & g \notin \Gamma(\tau)^{\perp} \\ \overline{\tau}\left(\pi_\tau(a)w_g\right), & \ g \in \Gamma(\tau)^{\perp} \ .  \end{cases}  
\end{equation*}

\end{proof}

\begin{lemma}\label{27-05-20d} Let $\tau \in \partial_e T_\beta(A)$. There is a representation $w^{\tau}$ of $\Gamma(\tau)^\perp$ by unitaries in $\pi_{\tau}(A)''$ such that
\begin{itemize}
\item  $\overline{\gamma}_g(w^\tau_h) = w^\tau_h$ and
\item $\overline{\gamma}_h = \Ad w^\tau_h$ 
\end{itemize}
for all $g \in G$ and all $h \in \Gamma(\tau)^\perp$.
\end{lemma}
\begin{proof} From the second item in Lemma \ref{13-05-20} we know that there is an $\omega_\tau \in \partial_e S_{\beta}(A \rtimes_\gamma G)$ such that ${\omega_\tau}|_A = \tau$. The desired representation arises then from an application of Lemma \ref{18-05-20}.
\end{proof}

\begin{lemma}\label{26-05-20b} Let $\tau \in \partial_e T_\beta(A)$ and let $\xi \in \widehat{\Gamma(\tau)^\perp}$ be a character on $\Gamma(\tau)^\perp$. Let $w^\tau$ be the representation from Lemma \ref{27-05-20d}. There is an extremal $\beta$-KMS state $\omega_\xi \in \partial_e S_\beta(A \rtimes_\gamma G)$ such that
 $$
 \omega_\xi(au_g)  = \begin{cases} 0 \ , & g \notin \Gamma(\tau)^{\perp} \\ \xi(g)\overline{\tau}\left(\pi_\tau(a)w^\tau_g\right), & \ g \in \Gamma(\tau)^{\perp} \  \end{cases}
 $$
for all $g \in G$ and all $a \in A$.
\end{lemma} 
 
\begin{proof}  We extend first $\xi$ to a character on $G$, which we also denote by $\xi$. Let $\delta: \widehat{G} \to \text{Aut} (A \rtimes_\gamma G )$ denote the dual action, i.e. $\delta_{\eta}(au_{g})=\eta(g) a u_{g}$ for $a\in A$ and $g\in G$. Let $\omega_\tau \in \partial_e S_\beta(A \rtimes_\gamma G)$ be the state from the proof of Lemma \ref{27-05-20d}. Since $\delta$ and $\alpha^{D}$ commute it is now straightforward to check that $\omega_\tau\circ \delta_{\xi}$ is a $\beta$-KMS state satisfying the formula in the lemma. Since $\omega_\tau$ is extremal and $\delta_{\xi}$ is an automorphism we also get $\omega_\tau\circ \delta_{\xi} \in \partial_e S_\beta(A \rtimes_\gamma G)$.

\end{proof}

\begin{thm}\label{26-05-20d} $X$ be the set of pairs $(\tau, \xi)$ where $\tau\in \partial_eT_\beta(A)$ and $\xi \in \widehat{\Gamma(\tau)^\perp}$. There is a bijection $\Phi : X \to \partial_e S_\beta(A \rtimes_\gamma G)$ such that 
$$
\Phi(\tau,\xi)(au_g) \ = \ \begin{cases} 0 \ , & g \notin \Gamma(\tau)^{\perp} \\ \xi(g)\overline{\tau}\left(\pi_\tau(a)w^\tau_g\right), & \ g \in \Gamma(\tau)^{\perp} \  \end{cases}
 $$
for all $g \in G$ and all $a \in A$, where $w^\tau$ is the representation from Lemma \ref{27-05-20d}.
\end{thm} 
\begin{proof} $\Phi$ is well-defined by Lemma \ref{26-05-20b}. To see that it is surjective let $\omega \in \partial_e S_{\beta}(A \rtimes_\gamma G)$ and let $\tau = \omega|_{A}$. We use first Lemma \ref{13-05-20} and Lemma \ref{18-05-20} to get a unitary representation $w$ of $\Gamma(\tau)^\perp$ with the properties stated there. Then $w^\tau_hw_h^*$ is in the center of $\pi_\tau(A)''$ and commutes in $\pi_{\omega_\tau}(A\rtimes_\gamma G)''$ with $\pi_{\omega_\tau}(u_g)$ for all $g \in G$ and $h\in \Gamma(\tau)^{\perp}$, where $\omega_\tau$ is the state giving us $w^\tau$, cf. Lemma \ref{27-05-20d}. It follows that $w^\tau_hw_h^*$ is in the center of $\pi_{\omega_\tau}(A\rtimes_\gamma G)''$ which is a factor since $\omega_\tau$ is extremal, cf. Theorem 5.3.30 in \cite{BR}. Hence $w_h = \xi(h)w^\tau_h$ for some $\xi(h) \in \mathbb T$. Note that $h \mapsto \xi(h)$ is character on $\Gamma(\tau)^\perp$ since $w^\tau$ and $w$ are both representations and observe that $\omega = \Phi(\tau,\xi)$. To see $\Phi$ is injective assume that $\Phi(\tau,\xi) = \Phi(\tau',\xi')$. Then $\tau = \Phi(\tau,\xi)|_A = \Phi(\tau',\xi')|_A = \tau'$ and $\overline{\tau}(\pi_{\tau}(a)\xi(h)w^\tau_h) = \overline{\tau}(\pi_{\tau}(a)\xi'(h)w^\tau_h)$ for all $a \in A$ and all $h \in \Gamma(\tau)^\perp$. Fix $h \in \Gamma(\tau)^\perp$ and choose a sequence $\{a_n\}_{n=1}^{\infty}$ in $A$ such that $\lim_{n\to \infty} \pi_\tau(a_n) = {w^\tau_h}^*$ strongly. Then
$$
\xi(h) = \lim_{n \to \infty} \xi(h)\overline{\tau}(\pi_{\tau}(a_n)w^\tau_h) =\lim_{n \to \infty} \xi'(h)\overline{\tau}(\pi_{\tau}(a_n)w^\tau_h) = \xi'(h) \ .
$$

\end{proof}

Theorem \ref{26-05-20d} gives a complete description only of the extremal $\beta$-KMS states for $\alpha^D$, but this is sufficient for most purposes since an arbitrary $\beta$-KMS state is the barycenter of a unique Borel probability measure concentrated on the extremal elements.

Using Bochners theorem we obtain from Theorem \ref{26-05-20d} the following corollary.

\begin{cor}\label{18-05-20i} Let $\tau \in \partial_e T_\beta(A)$. There is an affine homeomorphism between $\left\{ \omega \in  S_{\beta}(A \rtimes_{\gamma} G) : \ \omega|_A = \tau \right\}$ and the set of normalized positive definite functions $\varphi$ on $\Gamma(\tau)^\perp$ given by the formula
$$
\omega_\varphi(au_g) \ = \ \begin{cases} 0 \  , &  \ g\notin \Gamma(\tau) ^{\perp} \\  \varphi(g)\overline{\tau}(\pi_{\tau}(a)  w^\tau_g)  \ , & \ g \in \Gamma(\tau) ^{\perp} \ . \end{cases}
$$
\end{cor}

\begin{cor}\label{27-05-20f} The following are equivalent:
\begin{itemize}
\item Then map $\omega \mapsto \omega|_A$ is an affine homeomorphism of $S_{\beta}(A \rtimes_\gamma G)$ onto $T_\beta(A)$.
\item $\Gamma(\tau) = \widehat{G}$ for all $\tau \in \partial_e T_\beta(A)$.
\item $\pi_\tau(A)''\rtimes_{\overline{\gamma}} G$ is a factor for all $\tau \in \partial_e T_\beta(A)$.
\end{itemize}
\end{cor}
\begin{proof} The first two items are equivalent by Theorem \ref{26-05-20d} and the last two by a combination of Lemma \ref{18-05-20a} above with Theorem 8.11.15 in \cite{Pe}.
\end{proof}

\begin{remark}\label{27-05-20g} It's time for the second author to make amends. It has been demonstrated by examples in \cite{U} that the statement of Theorem 4.3 in \cite{Th} is not correct. To get a correct statement the fourth condition must be removed; it does imply the other conditions but not vice versa. Similarly, and for the same reason, the third condition in Theorem 5.4 in \cite{Th} must be removed. The mistake arises from the fact that Lemma 4.2 in \cite{Th}, which was uncritically gleaned from a paper by Bedos, \cite{Be}, is in fact not true. That the two theorems, 4.3 and 5.4 in \cite{Th}, are correct when the last condition is removed in both, follows from Corollary \ref{27-05-20f} by taking $\beta = 0$ or $D =0$. 
\end{remark}

\begin{example}\label{13-06-20} Let $\mathcal A$ be a finitely generated abelian group and $V_a, a \in \mathcal A$, a faithful representation of $\mathcal A$ in the unitary group $U_k$ of $M_k(\mathbb C)$. Let $n_0$ be the least natural number such that $2^{n_0} \geq k$. When $n \geq n_0$, set
$$
V^n_a = \diag (V_a, 1,1, \cdots , 1) \ ,
$$
giving us a faithful unitary representation of $\mathcal A$ in $U_{2^n}$.  We consider the CAR algebra $A$ as the infinite tensor product
$$
A = \otimes_{n=n_0}^{\infty} M_{2^{n}}(\mathbb C) \ .
$$ 
Set
$$
W^l_a = V^{n_0}_a \otimes V^{n_0+1}_a \otimes \cdots \otimes V^l_a \otimes 1 \otimes 1 \otimes 1 \otimes \cdots \ \in \ A \ .
$$
We can then define an action $\eta: \mathcal A \to \Aut A$ such that
$$
\eta_a(x) = \lim_{l \to \infty} W^l_ax{W^l_a}^* \ .
$$ 
It follows from Theorem 3.2 in \cite{L} that $\eta_a$ is an outer automorphism for $a \neq 0$. On the other hand if $\tau$ denotes the trace state of $A$ we find that
$$
\left\|W^n_a - W^{n+1}_a\right\|_\tau = \tau\left((W^n_a - W^{n+1}_a)^*(W^n_a - W^{n+1}_a)\right)^{\frac{1}{2}} \leq \frac{\|V_a -1\|}{\sqrt{2}^{n+1}} \ ,
$$
which implies that $\{\pi_\tau(W^n_a)\}$ converges in the strong operator topology to a unitary in $\pi_\tau(A)''$. This shows that $\overline{\eta}$ is an inner action on $\pi_\tau(A)''$. Let $R_\theta$ be an irrational rotation of the circle; i.e. $\theta\in \mathbb R\backslash \mathbb Q$ and $R_\theta(z) = e^{2\pi i \theta}z$. We define an action $\gamma : \mathcal A \times  \mathbb Z \to \Aut C(\mathbb T)\otimes A$ such that
$$
\gamma_{(a,z)} (f)(y) =  \eta_a \left( f\left( R^{-z}_\theta(y)\right) \right) \ 
$$
when $f : \mathbb T \to A$ is continuous. Set
$$
B = \left( C(\mathbb T)\otimes A\right)\rtimes (\mathcal A \times \mathbb Z) \ .
$$
Then no non-trivial ideal is fixed by $\gamma$ and $\gamma_{(z,a)}$ is outer for all $(z,a) \neq (0,0)$. It follows therefore from \cite{Ki} that $B$ is simple.
Let $F : \mathbb T \to \mathbb R$ be a continuous function, considered as a self-adjoint element of $C(\mathbb T) \subseteq C(\mathbb T)\otimes A$; the center of $C(\mathbb T)\otimes A$.  Set 
\begin{equation}\label{12-06-20a}
S_k(F) = \begin{cases} -\sum_{j=1}^{|k|} F\circ R_{\theta}^{-j} \ , & k \leq -1 \ , \\ 0 \ , \ & k = 0  \ ,\\ \sum_{j=0}^{k-1} F \circ R_{\theta}^j \ , \ & k \geq 1 \ . \end{cases}
\end{equation}
We can then define $D :  \mathcal A \times \mathbb Z \to C(\mathbb T)$ such that
$$
D_{(a,z)} = S_{z}(F) \ .
$$
Then $D$ satisfies the cocycle relation \eqref{05-05-20} and we obtain therefore a flow $\alpha^D$ on $B$. To determine the $\beta$-KMS states for $\alpha^D$ we observe first that the set of trace states $T(C(\mathbb T) \otimes A)$ can be identified with the simplex of Borel probability measures on $\mathbb T$; any such measure $m$ defines a trace state $\tau_m$ on $C(\mathbb T) \otimes A$ such that
$$
\tau_m (f) \ = \ \int_{\mathbb T} \tau(f(t)) \ \mathrm{d}m(t) 
$$
when $f : \mathbb T  \to A$ is continuous and $\tau$ is the trace on $A$. By Lemma \ref{05-05-20b} $\tau_m$ is the restriction to $C(\mathbb T) \otimes A$ of a $\beta$-KMS state for $\alpha^D$ iff 
$$
\int_{\mathbb T} \tau(\gamma_{(a,z)}(f)(t)) \ \mathrm{d}m(t) \ = \ \int_{\mathbb T} \tau (f(t)) e^{\beta S_z(F)(t)} \ \mathrm{d} m(t)
$$
for all $f \in C(\mathbb T) \otimes A$ and all $z \in \mathbb Z$ and $a\in \mathcal{A}$. By definition of $\gamma$ and $D$ this holds if and only if
$$
\int_{\mathbb T} g \circ R_{\theta}^{-k}(t) \ \mathrm{d}m(t) \ = \ \int_{\mathbb T} g(t) e^{\beta S_k(t)} \ \mathrm{d} m(t)
$$
for all $g \in C(\mathbb T)$ and all $k \in \mathbb Z$. This, in turn, holds if and only if $m$ is $e^{\beta F}$-conformal for $R_\theta$ in the terminology of \cite{CT}. We can therefore use Theorem 6.9 in \cite{CT} to conclude that $T_\beta( C(\mathbb T)\otimes A)$ is non-empty for all $\beta \in \mathbb R $ if $\int_\mathbb T F(t) \ \mathrm{d} t = 0$, where $\mathrm{d}t$ is the Lebesgue measure on $\mathbb T$, while $T_{\beta}( C(\mathbb T)\otimes A) = \emptyset$ for $\beta \neq 0$ if $\int_\mathbb T F(t) \ \mathrm{d} t \neq 0$. In particular, there are no $\beta$-KMS states for $\alpha^D$ when $\beta \neq 0$ if $\int_\mathbb T F(t) \ \mathrm{d} t \neq 0$. We assume therefore in the following that $\int_\mathbb T F(t) \ \mathrm{d} t = 0$. It follows from Theorem 7.1 in \cite{CT} that the set of $e^{\beta F}$-conformal measures can contain arbitrarily many $R_\theta$-ergodic measures in this case. By Theorem \ref{26-05-20d} each such measure $m$ gives rise to a closed face in the simplex $S_\beta (B)$ of $\beta$-KMS states for $\alpha^D$, whose set of extreme points is in bijective correspondence with the set 
$$
\widehat{\Gamma(\tau_m)^\perp} \ .
$$  
Note that $\pi_{\tau_m}( C(\mathbb T)\otimes A)''$ is isomorphic to the von Neumann algebra tensor product $L^{\infty}(\mathbb T,m)\otimes \pi_\tau(A)''$ and that the action $\overline{\gamma}$ respects this tensor product decomposition. Since $\overline{\eta}$ is inner on $\pi_\tau(A)''$ it follows that $\overline{\gamma}_{(a,z)}$ is inner iff $z= 0$, in which case $\overline{\gamma}_{(a,z)} = \Ad (1 \otimes U)$ for some unitary $U \in \pi_\tau(A)''$ which is fixed by $\overline{\eta}$. It follows therefore from Theorem 8.9.4 in \cite{Pe} that ${\Gamma(\tau_m)^\perp} \simeq \mathcal A$. We conclude then from Theorem \ref{26-05-20d} that $\partial_e S_\beta (B)$ is in bijection with the set 
$$
\partial_e T_\beta( C(\mathbb T)\otimes A ) \times \widehat{\mathcal A} \ ,
$$ 
where $\partial_e T_\beta( C(\mathbb T)\otimes A )$ can be identified in a natural way with the set of $R_\theta$-ergodic $e^{\beta F}$-conformal measures. By taking $\beta = 0$ in Theorem \ref{26-05-20d} we find that the extremal boundary of the simplex of trace states of $B$ is homeomorphic to $\widehat{\mathcal A}$ and hence that the traces states of $B$ are in bijective correspondence with the Borel probability measures on $\widehat{\mathcal A}$.

\end{example}

\end{document}